\documentclass[12pt]{article}
\usepackage{amsmath,amsfonts,amscd,amsthm}
\usepackage{srcltx}
\usepackage{multicol}
\usepackage[brazil,english]{babel}
\usepackage[latin1]{inputenc}
\newcommand{\R}{\mathbb{R}} 
\newcommand{\N}{\mathbb{N}} 

\usepackage{tikz}
\usepackage{pgfplots}
\pgfplotsset{width=6cm,compat=newest}
\usetikzlibrary{decorations.markings}
\tikzset{->-/.style={decoration={
  markings,
  mark=at position #1 with {\arrow{stealth}}},postaction={decorate}}}

\tikzset{-<-/.style={decoration={
  markings,
  mark=at position #1 with {\arrow{stealth reversed}}},postaction={decorate}}}

\usetikzlibrary{matrix}

\newtheorem{theorem}{Theorem}[section]
\newtheorem{corollary}[theorem]{Corollary}
\newtheorem{lemma}[theorem]{Lemma}

\theoremstyle{definition}
\newtheorem{definition}[theorem]{Definition}
\newtheorem{remark}[theorem]{Remark}

\numberwithin{equation}{section}

\newcommand{\teoref}[1]{Theorem \ref{#1}}
\newcommand{\lemaref}[1]{Lemma \ref{#1}}

\newcommand{\corref}[1]{Corollary \ref{#1}}

\newcommand{\obsref}[1]{Remark \ref{#1}}

\begin{document}

\title{Weak topological conjugacy via character of recurrence on impulsive dynamical systems}

\author{E. M. Bonotto\thanks{Instituto de Ci\^encias Matem\'aticas e
de Computa\c{c}\~ao, Universidade de S\~ao Paulo-Campus de S\~ao
Carlos, Caixa Postal 668, 13560-970 S\~ao Carlos SP, Brazil.
E-mail: ebonotto@icmc.usp.br. Partially supported by FAPESP 2016/24711-1 and CNPq 310497/2016-7.}\,
, D. P. Demuner \thanks{Universidade Federal do Esp\'irito Santo, Vit\'oria ES, Brazil. E-mail: ddemuner@gmail.com.} \ and G.
M. Souto\thanks{Universidade Federal do Esp\'irito Santo, Vit\'oria ES, Brazil.
E-mail: ginnaramsouto@gmail.com.} }

\date{}
\maketitle

\begin{abstract} In the present paper, we define the concept of weak topological conjugacy and we establish sufficient conditions to obtain this kind of topological conjugacy between two limit sets. We use the character of recurrence to obtain the results.
\end{abstract}

\section{Introduction}

\hspace{.5cm}
When dealing with continuous or discontinuous dynamical systems, it is natural to ask if two dynamical systems possess, in some sense, orbits with the same structure. These types of dynamical systems with similar appearance are called equivalent systems. There exist several notions of equivalence which depend on the required smoothness of the systems, such as the homomorphism. Two dynamical systems $(X,\pi)$ and $(Y,\sigma)$
are homomorphic if there is a continuous mapping $h:X\to Y$ such that the following diagram
\begin{figure}[h]
\begin{center}
\begin{tikzpicture}
  \matrix (m) [matrix of math nodes,row sep=3em,column sep=4em,minimum width=2em]
  {
     X\times \mathbb{R}_{+} & X \\
     Y\times \mathbb{R}_{+} & Y \\};
  \path[-stealth]
    (m-1-1) edge node [left] {$h\times I$} (m-2-1)
            edge node [below] {$\pi$} (m-1-2)
    (m-2-1.east|-m-2-2) edge node [below] {$\sigma$}
            node [above] {} (m-2-2)
    (m-1-2) edge node [right] {$h$} (m-2-2)
            ;
\end{tikzpicture}
\end{center}
\end{figure}\\
is commutative, that is, $h(\pi(x,t))=\sigma(h(x),t)$ for all $(x,t)\in X\times\mathbb{R}_{+}$ ($h$ maps orbits of $\pi$ to orbits of $\sigma$ homomorphically and preserving orientation of the orbits). If $h$ is a homeomorphism then $(X,\pi)$ and $(Y,\sigma)$ are called
topologically conjugate. Many results on equivalence were explored for continuous dynamical systems, see for instance \cite{Irwin, Katok, Smalle}. However, this concept has not been much investigated for impulsive systems.

The theory of impulsive dynamical systems is an important
tool to describe the evolution of systems where the continuous
development of a process is interrupted by abrupt changes of
state, whose duration is negligible in comparison with the duration of entire evolution processes. Many real world applications in biological phenomena, engineer,  physics,
optimal control model and frequency modulated systems are described by impulsive systems, see for example the references \cite{Ambrosino, Cortes, El, Liang, Riva, Yuan}.

One of the first works on equivalence for impulsive dynamical systems was carried out in \cite{BonottoFederson1}, where the authors deal with topological conjugate and give some results on the structure of the phase space. We present in this paper a study of topological conjugacy using character of recurrence in the sense of \cite{Cheban}.

We may observe that there is not a relation of topological conjugacy between  continuous dynamical systems and impulsive dynamical systems. For instance, consider the dynamical systems $(\mathbb{R}^{2},\pi)$
and 
$(\mathbb{R}^{2},\sigma)$ presented in Figure $1$, which are not homeomorphic.
\begin{figure}[htp]\label{Figura2}\begin{center}
\begin{tikzpicture}
\draw[->-=1] (-8,0) -- (-5,0); 
\draw[->-=1] (-8,0.5) -- (-5,0.5); 
\draw[->-=1] (-8,1) -- (-5,1); 
\draw[->-=1] (-8,-0.5) -- (-5,-0.5); 
\draw[->-=1] (-8,-1) -- (-5,-1); 
\draw(-6.5,-1.5)node[below]{$\pi((x,y),t)=(x+t,y)$};
\draw[->-=0.5,black] (1,1) -- (0,0); 
\draw[->-=0.5,black] (-1,1) -- (-0.1,0.1); 
\draw[dashed,black] (-0.1,0.1) -- (0,0); 
\draw[->-=0.5,black] (1,-1) -- (0,0); 
\draw[->-=0.5,black] (-1,-1) -- (0,0); 
\draw[->-=0.5,black] (1.4,0) -- (0,0); 
\draw[->-=0.5,black] (-1.4,0) -- (0,0); 
\draw[->-=0.5,black] (0,1.4) -- (0,0); 
\draw[->-=0.5,black] (0,-1.4) -- (0,0); 
\draw[color=black,fill] (0,0) circle (1pt);
\draw(0,-1.5)node[below]{$\sigma((x,y),t)=(xe^{-t},ye^{-t})$};
\draw[->,thick] (-4,0) -- (-2,0); 
\draw[thick] (-3.1,-0.1) -- (-2.9,0.1); 
\end{tikzpicture}\end{center}
 \caption{Systems are not homeomorphic.
 }\label{figexem1}
 \end{figure}
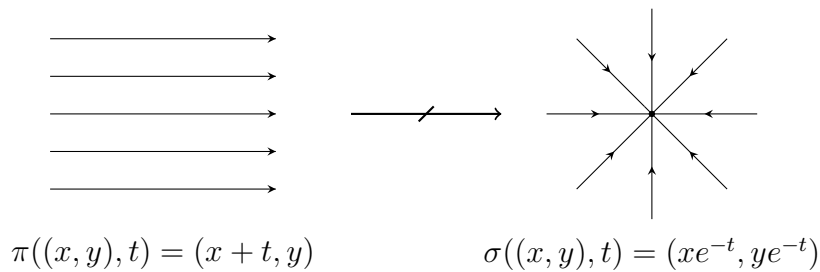
Moreover,  it is not possible to obtain a topological conjugation restricted on any invariant subset from the phase space.
However, is it possible to perturbed the systems $(\mathbb{R}^{2},\pi)$
and 
$(\mathbb{R}^{2},\sigma)$ under impulse conditions to obtain a topological conjugacy restricted on an invariant set as we will see in Subsection $4.1$.

In this work, we introduce the notion of weak topological conjugation and comparability of points on impulsive dynamical systems. We start in Section 2 by presenting
the basis of the theory of dynamical systems with impulses. In Section 3, we present some additional definitions and auxiliary results. Section 4 concerns with the main results. This section is divided in two
parts. In Subsection 4.1, we present the concept of a weak topological conjugation and the notion of comparable points via character of recurrence. 
Using the comparability by the character of recurrence, we show the existence of a continuous function that maps orbits between two limit sets homomorphically and preserving orientation, see Theorem \ref{T2.1} and Corollary \ref{CT2.1}. As a consequence, in Corollary \ref{CT2.2}, we establish sufficient conditions for the existence of a weak topological conjugation between two limit sets. In Subsection 4.2, we consider asymptotically almost periodic motions and we investigate some relations between comparable points and the existence of a weak topological conjugation, see Theorem \ref{T2.2}, Theorem \ref{T2.3} and Corollary \ref{C3.12}. In Theorem \ref{T3.15}, we present sufficient conditions for two points to be comparable in limit.

\section{Preliminares}

\hspace{0.5cm} Let $(X, \rho)$ be a metric space and
$\mathbb{R}_{+}$ be the set of all non-negative real
numbers. The triple $( X, \pi, \mathbb{R}_{+})$ is
called a \emph{semidynamical system} on $X$ if the mapping $\pi :  X
\times \mathbb{R}_+\to  X$ is continuous, $\pi (x,0)=x$ and $\pi
(\pi (x,t),s)= \pi (x, t+s)$ for all $x \in
 X$ and $t,s\in \mathbb {R}_+$.   

Along to this text, we shall denote the system $( X,\pi,
\mathbb{R}_{+})$ simply by $(X, \pi)$. For
every $x \in X$, we consider the continuous function $\pi_{x}:
\mathbb{R}_+ \to X$ given by $\pi_{x}(t)=\pi(x,t)$ which is called 
the \emph{motion} of $x$. The \emph{positive
orbit} of $x\in X$ is given by $\pi^{+} (x) = \{\pi (x,t) : t \in
\mathbb{R}_{+}\}$. For $t \geq 0$ and $x \in  X$, we
define $F(x,t)= \{y \in X: \pi (y,  t) = x \}$ and, for $\Delta
\subset [0,  +\infty)$ and $D \subset  X$, we write
$$
F(D,  \Delta) =  \{F(x,  t): x \in D
 \text{ and }  t \in \Delta \}.
$$
A point $x \in  X$ is called an \emph{initial point}, if
$F(x, t) = \emptyset$ for all $t>0$.

An
\emph{impulsive semidyna\-mical system} (ISS), represented by  $( X,  \pi; M, I)$,
consists of a semidynamical system $( X,  \pi)$, a nonempty closed
subset $M$ of $X$ such that for every $x \in  M$, there exists
$\epsilon_x >0$ such that
\[
F(x,  (0,  \epsilon_x))\cap  M = \emptyset \quad  \text{and}
\quad \pi (x,  (0, \epsilon_x))\cap  M = \emptyset,
\] and a continuous function $I :  M \to X$ whose
action we explain below in the description of an impulsive
trajectory.  The set $M$ is called the \emph{impulsive set} and
the function $I$ is called the \emph{impulse function}. We also
define
$$
M^{+}(x)=\Big(\bigcup_{t>0} \pi(x,t)\Big)\cap  M \quad \text{for all} \quad x \in X.
$$

If $x \in  X$ and $
M^{+}(x) \neq \emptyset$, then there is a number $s > 0$ such
that $\pi (x,  t) \notin  M$ for $0 < t < s$ and $\pi (x,
s) \in  M$. In this way, it is possible to define a
function $\phi :  X \to (0,+\infty]$ in the following manner
\begin{equation}\label{phi}
 \phi(x)=\begin{cases}
  s, & \text{if $\pi (x,  s) \in M$ and
  $\pi(x,  t) \notin M$ for $0 < t < s$,} \\
  +\infty, & \text{if } M^{+}(x) = \emptyset .
\end{cases}
\end{equation}
The number $\phi(x)$ represents the least positive time for
which the trajectory of $x\in X$ meets $M$ when $\phi (x) < +\infty$.

The \emph{impulsive trajectory} of $x$ in $(X,  \pi; M, I)$ is an
$ X-$valued function $\tilde{\pi } _{x}$ defined in some subset
$[0,s)$ of $\mathbb{R}_{+}$ $(s$ may be $+\infty)$. The
description of such trajectory follows inductively as described in
the next lines.

If $ M^{+} (x) = \emptyset$, then $\phi (x) = +\infty$ and $\tilde{\pi}_{x}(t) = \pi
(x,t)$ for all $t \in \mathbb{R}_{+}$.
However, if $ M^{+} (x) \neq \emptyset$ then $\phi (x) = s_{0} < +\infty$, $\pi (x,  s_{0}) = x_{1} \in  M$ and $\pi (x,  t) \notin
M$ for $0< t <s_{0}$. Thus we define $\tilde{\pi}_{x}$ on
$[0,s_{0}]$ by
\[
 \tilde{\pi}_{x} (t)=\begin{cases}
  \pi (x,  t), & 0 \leq t < s_{0}, \\
  {x^{+}_{1}}, & t = s_{0},
\end{cases}
\]
where $x^{+}_{1} = I(x_1)$. Let us
denote $x$ by $x_0^+$.

Since $s_{0} < + \infty$, the process now continues from
$x_{1}^{+}$ onwards. The reader may consult \cite{BonottoFederson1, Ciesielski2} for more details.

 Notice that $\tilde{\pi}_{x}$ is defined on
each interval $[t_{n}(x),  t_{n+1}(x)]$, where $t_0(x) = 0$ and $t_{n+1}(x) =
\displaystyle\sum _{i=0}^{n} s_{i}$, $n=0,1,2,\ldots$. If $ M^{+}(x^{+}_{n}) \neq \emptyset$ for all $
n=0,1,2,\dots$, then $\tilde{\pi}_{x}$ is defined
on the interval $[0,  T(x))$, where $T(x) = \displaystyle\sum
_{i=0}^{\infty} s_{i}$ $($recall that $x_0^+=x$ and $x_{n}^{+}= I(x_n) = I(\pi(x_{n-1}^{+},s_{n-1}))$ for $n=1,2,\ldots)$.

The \emph{impulsive positive orbit} of a point $x\in X$ in
$(X, \pi; M,  I)$ is defined by the set
\[
\tilde{\pi}^{+}(x)= \{\tilde{\pi}(x,  t): t \in [0,
T(x))\}.
\]

Analogously to the non-impulsive case, an ISS satisfies the following
 standard properties: $\tilde{\pi}(x,  0) = x$  and
 $\tilde{\pi}(\tilde{\pi}(x,  t),  s)
= \tilde{\pi}(x,  t+s)$, for all $x\in X$ and $ t, s \in [0,  T(x))$ such
that $t+s \in [0,  T(x))$. For more details about the theory of impulsive systems, the reader may consult
\cite{BonottoFederson1, artigo2, BonottoMatheus, Ciesielski2, Ciesielski3, Kaul}.

Now, let us discuss the continuity of the
function $\phi$ presented in \eqref{phi}. The continuity of $\phi$ is studied in
 \cite{Ciesielski2}.

Let $( X,\pi)$ be a semidynamical system. Any closed set $
S\subset X$ containing $x$ $(x \in  X)$ is called a \emph{section}
or a \emph{$\lambda $-section} through $x$, with $\lambda>0$, if
there exists a closed set $L\subset X$ such that
\begin{itemize}
\item[$a)$] $F(L,  \lambda) = S$;
\item[$b)$] $F(L,  [0,  2 \lambda])$ is a neighborhood of $x$;
\item[$c)$] $F(L,  \mu) \cap F(L,\nu) = \emptyset$, for
$0 \leq \mu < \nu \leq 2 \lambda$.
\end{itemize}
The set $F(L,[0,  2 \lambda])$ is called a \emph{tube} or a
\emph{$\lambda$-tube} and the set $L$ is called a \emph{bar}.

Any tube $F(L,[0,2 \lambda])$ given by a section $S$ through $x
\in X$ such that \linebreak$S \subset M \cap F(L,  [0,2 \lambda])$ is called
\emph{TC-tube} on $x$. We say that a point $x\in  M$ fulfills the
\emph{Tube Condition} and we write TC, if there exists a TC-tube
$F(L,  [0,  2 \lambda])$ through $x$. In particular, if $S = M
\cap F(L,  [0,  2 \lambda])$ we have a \emph{STC-tube} on $x$ and
we say that a point $x \in  M$ fulfills the \emph{Strong Tube
Condition} (we write STC), if there exists a STC-tube $F(L,  [0,
2 \lambda])$ through $x$.

Theorem \ref{theorem2.1} concerns the continuity of $\phi$ which is
accomplished outside of $M$.

\begin{theorem}\rm\cite[Theorem 3.8]{Ciesielski2}\label{theorem2.1}\it \,
Consider an impulsive semidynamical system $(X,  \pi; M,  I)$.
Assume that no initial point in $( X,  \pi)$ belongs to the impulsive
set $M$ and that each element of $M$ satisfies the condition $TC$.
Then $\phi$ is continuous at $x$ if and only if $x \notin  M$.
\end{theorem}

\section{Additional definitions and auxiliary results}

\hspace{.5cm} Let $(X,\pi; M,I)$ be an ISS. Throughout this paper, we shall assume the following conditions:
\begin{enumerate}\item[(H1)] No initial point in $(X,\pi)$ belongs to the impulsive set $M$ and each element of $M$ satisfies the condition STC, consequently $\phi$ is continuous on $X\setminus M$;
\item[(H2)] $M\cap I(M)=\emptyset$;
\item[(H3)] For each $x\in X$, the motion $\tilde{\pi}(x, t)$ is defined for every $t\geq 0$.
\end{enumerate}

Given $A\subset X$ and $\Delta\subset
\mathbb{R}_{+}$, we denote
$$\tilde{\pi}(A, \Delta) = \{\tilde{\pi}(x, t): \, x \in A, \, t \in
\Delta\} \quad \text{and} \quad \tilde{\pi}^+(A) = \displaystyle\bigcup_{x\in A}\tilde{\pi}^+(x).$$

A set $A$ is called \emph{positively $\tilde{\pi}$-invariant} if $\tilde{\pi}^+(A)\subset A$. Also, if $I(A\cap M) \subset A$ then we say that $A$ is $I$-invariant.

The \emph{positive limit set} of a point $x \in X$ in $(X, \pi; M, I)$ is given by
\[
\tilde{L}^{+}(x)=\{y\in X: \; \text{there is a sequence} \;
\{\lambda_n\}_{n \in \mathbb{N}} \subset \mathbb{R}_{+} \quad \text{such
that}
\]
\[
\quad\qquad \lambda_{n}\stackrel{n\to+\infty}{\longrightarrow}+\infty
\quad\mbox{and}\quad
\tilde{\pi}(x, \lambda_{n})\stackrel{n\to+\infty}{\longrightarrow}
y\}.
\]

By \cite[Proposition 4.3]{Manuel}  the set $\tilde{L}^{+}(x)\setminus M$ is positively $\tilde{\pi}$-invariant for all $x\in X$.

Next, we exhibit three auxiliary results concerning convergence on impulsive systems.

\begin{lemma}\label{lema1}\rm\cite[Lemma 3.8]{BonottoMatheus} \it
Let $x\notin M$ and $\{x_n\}_{n\in \mathbb{N}}$ be a sequence in $X\setminus M$ such that $x_n\stackrel{n\to+\infty}{\longrightarrow} x$. Then if $\alpha_n \stackrel{n\to+\infty}{\longrightarrow} 0$ and
$\alpha_n\geq 0$, for all $n\in \mathbb{N}$, we have $\tilde{\pi}(x_n, \alpha_n)\stackrel{n\to+\infty}{\longrightarrow} x$.
\end{lemma}

\begin{lemma}\label{lema2}\rm\cite[Corollary 3.9]{BonottoMatheus} \it
Let  $\{x_{n}\}_{n\in \mathbb{N}}$ be a sequence in $X$
which converges to  $x\in X\setminus M$. Then, given $t\geq 0$, there
is a sequence $\{\epsilon_{n}\}_{n\in \mathbb{N}} \subset \mathbb{R}_{+}$ such
that  $\epsilon_{n} \stackrel{n\to+\infty}{\longrightarrow} 0$ and  $\tilde{\pi}(x_{n},
t+\epsilon_{n})\stackrel{n\to+\infty}{\longrightarrow} \tilde{\pi}(x,t)$. \end{lemma}

\begin{lemma}\label{lema3}\rm\cite[Lemma 2.4]{artigo2} \it
Let $x\in X\setminus M$ and $\{x_{n}\}_{n\in \mathbb{N}}\subset X$ be a sequence which converges to $x$. Given $t \geq 0$ such that $t \neq
t_{k}(x)$, $k =0,1,2,\ldots$,  and $\{\lambda_{n}\}_{n\in \mathbb{N}} \subset \mathbb{R}_{+}$ is a sequence with  $\lambda_n \stackrel{n\to+\infty}{\longrightarrow} t$ then
 $\tilde{\pi}(x_n, \, \lambda_n) \stackrel{n \rightarrow
+\infty}{\longrightarrow} \tilde{\pi}(x, \, t).$
\end{lemma}

Lemma \ref{jaqueline} below establishes sufficient conditions for a limit set to contain points outside from $M$.

\begin{lemma} \rm\cite[Lemma 4.15]{jaqueline}\label{jaqueline} \it Let $x\in X$ be a point such that $\tilde{L}^{+}(x)\neq \emptyset$, then \linebreak$\tilde{L}^{+}(x)\setminus M\neq \emptyset$.
\end{lemma} 

A point $x\in X$ is called \emph{stationary}
with respect to $(X,\pi;M,I)$, if $\tilde{\pi}(x,t)=x$
\mbox{\mbox{for all} }  $t\geq 0$. If $\tilde{\pi}(x,\tau)=x$
for some $\tau> 0$, then $x$ will be called
\emph{$\tilde{\pi}$-periodic}. The point $x\in X$ is called \emph{positively Poisson $\tilde{\pi}$-stable} if $x \in \tilde{L}^+(x)$.

A point $x\in X$  is said to be \emph{almost
$\tilde{\pi}$-periodic} if for every $\epsilon>0$, there exists a $T=T(\epsilon)>0$ such that for every $\alpha \geq 0$, the interval $[\alpha,\alpha+T]$ contains a number
$\tau_{\alpha}>0$ such that
\[
\rho(\tilde{\pi}(x,t+\tau_{\alpha}), \tilde{\pi}(x,t))<\epsilon \quad \mbox{\mbox{for all} }  \; t\geq 0.
\]
The set $\{\tau_{\alpha}: \, \alpha \geq 0\}$ is called a \textit{family of almost period} of $x$. The reader may consult \cite{artigo2, Manuel} for more details about the theory of 
$\tilde{\pi}$-periodic motions.

\begin{lemma} \rm\cite[Lemma 4.22]{Manuel}\label{EM1} \it If $x\in X$ is almost $\tilde{\pi}$-periodic, then every point $y \in \tilde{\pi}^+(x)$ is also almost $\tilde{\pi}$-periodic.
\end{lemma}

\begin{theorem} \rm\cite[Theorem 3.2]{artigo2}\label{EG1} \it If $x\in X$ is almost $\tilde{\pi}$-periodic, then every point $y \in \overline{\tilde{\pi}^+(x)}\setminus M$ is almost $\tilde{\pi}$-periodic. Moreover, if $\{\tau_{\alpha}: \, \alpha \geq 0\}$ is a family of almost period of $x$ then  $\{\tau_{\alpha}: \, \alpha \geq 0\}$ is also a family of almost period for each $y \in \overline{\tilde{\pi}^+(x)}\setminus M$.
\end{theorem}

\begin{theorem} \rm\cite[Theorem 3.9]{artigo2}\label{EG2} \it If $x\in X$ is almost $\tilde{\pi}$-periodic, then $\tilde{L}^+(x) = \overline{\tilde{\pi}^+(x)}$. Moreover, $x$ is positively Poisson $\tilde{\pi}$-stable. 
\end{theorem}

\section{The main results}

\hspace{.5cm} This section concerns with the main results. We present a characterization for a new class of homeomorphic sets on impulsive systems via character of recurrence.
We shall consider throughout this section two impulsive semidynamical systems $(X,\pi; M_{X}, I_{X})$
and $(Y,\sigma; M_{Y}, I_{Y})$ satisfying the conditions (H1), (H2) and (H3) presented in Section 3, where $(X, \rho_X)$ and $(Y, \rho_Y)$ are metric spaces.

\subsection{Weak topological conjugation}

\hspace{.5cm} Let $(\mathbb{R}^2,\pi;M_{1},I_{1})$ be an ISS such that $$\pi((x_1, x_2),t)=(x_1+t, x_2)$$ for all $(x_1, x_2) \in \mathbb{R}^2$ and $t\geq 0$, $M_{1}=\{(x_1,x_2)\in \mathbb{R}^{2}: x_1=1\}$ and  $I_{1}:M_{1}\to \mathbb{R}^2$ is given by
$I_{1}(x_1, x_2)=\left(0,\frac{x_2}{2}\right)$ for all $(x_1, x_2) \in M_{1}$. Also, let $(\mathbb{R}^2, \sigma; M_{2},I_{2})$ be an ISS such that $$\sigma((y_1,y_2),t)=(y_1e^{-t},y_2e^{-t})$$ for all $(y_1, y_2) \in \mathbb{R}^2$ and $t\geq 0$, $M_{2}=\{(y_1,y_2)\in \mathbb{R}^{2}: y_{1}^2+y_{2}^2=e^{-2}\}$ and $I_{2}: M_{2}\to \mathbb{R}^2$ is given as follows: given $(y_1,y_2)\in M_2$ we consider the line
segment $\Gamma_{(y_1,y_2)}$ that connects the points $(y_1,y_2)$ and $(1,y_2)$. The point $I_2(y_1,y_2)$ is
the point in the intersection $\Gamma_{(y_1,y_2)}\cap \{(y_1,y_2)\in \mathbb{R}^{2}: y_{1}^2+y_{2}^2=1\}$, as we can see in Figure 2 below.

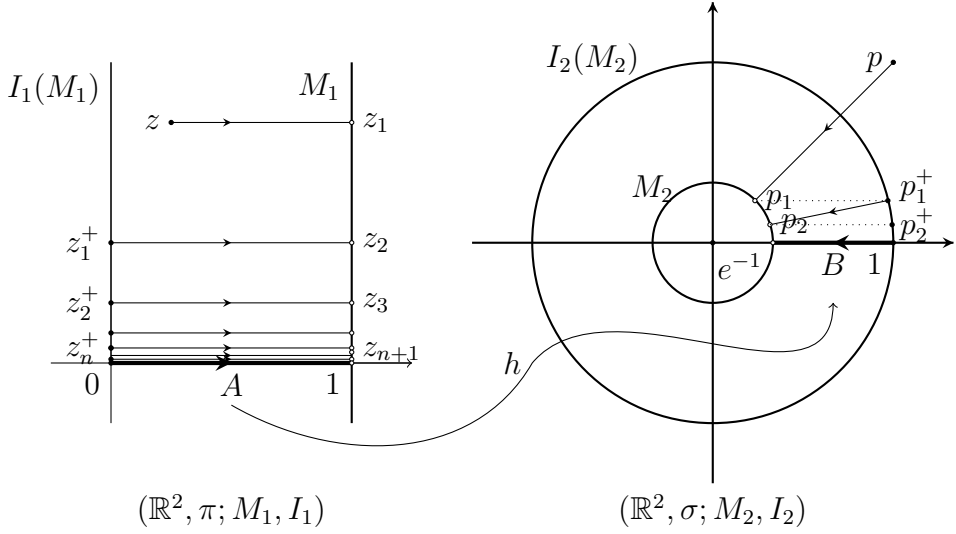
\begin{figure}[htp]\label{Figura2}\begin{center}
\begin{tikzpicture}[scale=0.8]

\draw[->] (0,0) -- (5,0); 
\draw (0,0) -- (0,5) node[anchor=north east]{$I_{1}(M_{1})$}; 
\draw[color=black,fill] (1,4) circle (1pt)node[left]{$z$}; 
\draw[thick](4,-1) -- (4,5)node[anchor=north east]{$M_{1}$}; 
\draw(-1,0) -- (5,0); 
\draw (0,-1) -- (0,5);
\draw[->-=0.333, black] (1,4)  -- (4,4); 
\draw[fill=white] (4,4) circle (1pt)node[right]{$z_1$};
\draw[->-=0.5,black](0,2) --  (4,2); 
\draw[fill=white] (4,2) circle (1pt)node[right]{$z_2$}; 
\draw[color=black,fill] (0,2) circle (1pt)node[left]{$z_1^{+}$}; 
\draw[->-=0.5,black](0,1) -- (4,1); 
\draw[fill=white] (4,1) circle (1pt)node[right]{$z_3$}; 
\draw[color=black,fill] (0,1) circle (1pt)node[left]{$z_2^{+}$}; 
\draw[->-=0.5,black] (0,0.5) -- (4,0.5); 
\draw[fill=white] (4,0.5) circle (1pt); 
\draw[color=black,fill] (0,0.5) circle (1pt); 
\draw[->-=0.5,black](0,0.25)  -- (4,0.25); 
\draw[fill=white] (4,0.25) circle (1pt); 
\draw[color=black,fill] (0,0.25) circle (1pt); 
\draw[->-=0.5,black](0,0.125) -- (4,0.125); 
\draw[fill=white] (4,0.18) circle (1pt)node[right]{$z_{n+1}$}; 
\draw[color=black,fill] (0,0.25) circle (1pt)node[ left]{$z_{n}^{+}$}; 

\draw[->-=0.5,black] (0,0.0625) -- (4,0.0625); 
\draw[fill=white] (4,0.0625) circle (1pt);
\draw[color=black,fill] (0,0.0625) circle (1pt);

 \draw[->-=0.5, ultra thick] (0,0) -- (4,0);
\draw(2,0)node[below]{$A$};
\draw[color=black,fill] (0,0) circle (1pt)node[anchor=north east]{$0$}; 
\draw[fill=white] (4,0) circle (1pt)node[anchor=north east]{$1$}; 
\draw[->-=1,thick] (10,-2) -- (10,6); 

\draw(2,-2)node[below]{$(\mathbb{R}^2,\pi;M_{1},I_{1})$};

\draw[->-=1, thick](6,2) -- (14,2); 

\draw[thick](10,2) circle [radius=1]; 
\draw(9,3.3)node[below]{$M_{2}$};

\draw[thick](10,2) circle [radius=3]; 
\draw(8,5.5)node[below]{$I_{2}(M_{2})$};

\draw[color=black,fill] (10,2) circle (1pt);
\draw[color=black,fill] (13,5) circle (1pt)node[left]{$p$}; 
\draw[->-=0.5,black] (13,5) -- (10.7,2.7); 
\draw[dotted]  (10.7,2.7) -- (12.9,2.7); 
\draw[fill=white] (10.7,2.7) circle (1pt)node[ right]{$p_1$};
\draw[color=black,fill] (12.91,2.7) circle (1pt); 
\draw(13.4, 3.4)node[below]{$p_1^+$};

\draw[->-=0.5,black] (12.9,2.7) -- (10.95,2.3); 
\draw[fill=white] (10.95,2.3) circle (1pt)node[ right]{$p_2$};
\draw[dotted] (10.95,2.3) -- (13.05,2.3); 
\draw[color=black,fill]  (12.98,2.3) circle (1pt); 
\draw(13.4, 2.75)node[below]{$p_2^+$};

 \draw[->-=0.5, ultra thick] (13,2) -- (11,2);
\draw(12,2)node[below]{$B$};
\draw(10,-2)node[below]{$(\mathbb{R}^2, \sigma;M_{2},I_{2})$};

\draw[->] (2,-0.7) to [out=330, in=240] (7,0)node[left]{$h$}  to [out=50, in=270] (12,1)
;  
\draw[fill=white] (11,2) circle (1pt)node[anchor=north east]{$e^{-1}$}; 
\draw[color=black,fill] (13,2) circle (1pt)node[anchor=north east]{$1$}; 
\end{tikzpicture}\end{center}
 \caption{Impulsive trajectories of systems $(\mathbb{R}^2,\pi;M_{1},I_{1})$ and $(\mathbb{R}^2,\sigma;M_{2},I_{2})$.}\label{figexem1}
 \end{figure}

As presented in the Introduction, in the absence of impulses, there is not a function that maps orbits of $\pi$ to orbits of $\sigma$ homeomorphically and preserving orientation of the orbits.
However, in the presence of impulses, let us consider the subsets $A=[0,1)\times\{0\}$ and $B=(e^{-1},1]\times \{0\}$ of $\mathbb{R}^{2}$. Note that $A$ is positively $\tilde{\pi}$-invariant and $B$ is positively $\tilde{\sigma}$-invariant. Define the mapping $h: A \to B$ by 
\[ 
h(x_1, 0)=(e^{-x_1},0),\quad (x_1,x_2)\in A.
\] 
Now, note that
\[
h(\tilde{\pi}((x_1, x_2),t))=\tilde{\sigma}(h(x_1, x_2),t) \quad \text{for all }   \quad (x_1, x_2)\in A\;\; \text{and} \; \;t\in\mathbb{R}_{+}.
\]
Thus, $h$ maps orbits of $\tilde{\pi}$ to orbits of $\tilde{\sigma}$ homeomorphically and preserving orientation of the orbits. The mapping $h$ is called a topological conjugacy between the sets $A$ and $B$. The reader may consult \cite{BonottoFederson1} for more details.

Since $\tilde{\pi}((1,0), t) = \pi((1,0), t) \notin \overline{A}$ for all $t > 0$, we may not define a topological conjugacy between the sets $\overline{A}$ and $\overline{B}$.
However,  note that $\tilde{\pi}^+(I_1(1, 0)) \subset A$ and
\[
h(\tilde{\pi}(I_{1}(1, 0),t))=\tilde{\sigma}(h(I_1(1, 0)),t)\,\quad \text{for all }t\in\mathbb{R}_{+}.
\]
The continuous mapping $h$ satisfies the properties:\vspace{0.1cm}

$i)$ $h(\tilde{\pi}((x_1, x_2),t))=\tilde{\sigma}(h(x_1, x_2),t)$ for all $(x_1, x_2)\in \overline{A}\setminus M_{1}$ and $t\in\mathbb{R}_{+}$;

$ii)$ $h(\tilde{\pi}(I_{1}(x_1, x_2),t))=\tilde{\sigma}(h(I_1(x_1, x_2)),t)$ for all $(x_1, x_2) \in \overline{A} \cap M_1$ and $t\in\mathbb{R}_{+}$.\vspace{0.1cm}

We will call this kind of mapping as a \textit{weak topological conjugation}. The next definition characterizes this type of conjugacy.

\begin{definition}\label{D4.1} Let $(X,\pi; M_{X}, I_{X})$ and $(Y,\sigma; M_{Y}, I_{Y})$ be two ISSs. The sets $A\subset X$ and $B\subset Y$ are weakly topologically conjugate, if there is a mapping $h:A\to B$ satisfying the following conditions:
\begin{itemize} 
\item[$a)$] $h$ is a homeomorphism;
\item[$b)$] $h(\tilde{\pi}(x,t))=\tilde{\sigma}(h(x),t)$ for all $x\in A\setminus M_{X}$ and $t\geq 0$;
\item[$c)$] $h(\tilde{\pi}(I_{X}(x),t))=\tilde{\sigma}(h(I_{X}(x)),t)$ for all $x\in  A\cap M_{X}$ and $t\geq 0$.
\end{itemize}
The mapping $h:A\to B$ is called a
\emph{weak topological conjugation}.
\end{definition}

Next, we define a special class of subsets of sequences in $\mathbb{R}_{+}$ with respect to an arbitrary impulsive system $(Z, \varrho;M_{Z},I_{Z})$.

\begin{definition}
 Let $z,p\in Z$ and $A\subset Z$ be a nonempty subset. We define the following sets:
 \begin{enumerate}
 \item[$i)$] $\mathcal{L}_{z,p}=\{\{t_{n}\}_{n\in\N}\subset\mathbb{R}_{+}:\displaystyle\lim_{n\to+\infty}\tilde{\varrho}(z,t_{n})=p\}$;
 \item[$ii)$] $\mathcal{L}_{z,p}^{+\infty}=\{\{t_{n}\}_{n\in\mathbb{N}}\in \mathcal{L}_{z,p}: t_{n} \stackrel{n\to+\infty}{\longrightarrow} +\infty\}$;
\item[$iii)$] $\mathcal{L}_{z}(A)=\cup\{\mathcal{L}_{z,p}: p\in A\}$ \; and \;
 $\mathcal{L}_{z}^{+\infty}(A)=\cup\{\mathcal{L}_{z,p}^{+\infty}: p\in A\}$;
 \item[$iv)$] $\mathcal{L}_{z}=\mathcal{L}_{z}(Z)$ \; and \; $\mathcal{L}_{z}^{+\infty}=\mathcal{L}_{z}^{+\infty}(Z)$. 
\end{enumerate}
\end{definition}
  
  Note that $\mathcal{L}_{z,p}^{+\infty} \subset \mathcal{L}_{z,p}$, $\mathcal{L}_{z}(A) \subset \mathcal{L}_{z}$ and $\mathcal{L}_{z}^{+\infty}(A) \subset \mathcal{L}_{z}^{+\infty}$. If $p \in A$ then
\[
\mathcal{L}_{z,p}^{+\infty} \subset \mathcal{L}_{z,p}\subset \mathcal{L}_{z}(A) \subset \mathcal{L}_{z} \quad \text{and} \quad \mathcal{L}_{z,p}^{+\infty} \subset \mathcal{L}_{z}^{+\infty}(A) \subset \mathcal{L}_{z}^{+\infty}.
\]

\begin{definition}
A point $y\in Y$ is called comparable with $x\in X$ by the character of recurrence with respect to a set $A\subset X$, or simply, comparable with $x$ with respect to a set $A$, if 
$\mathcal{L}_{x}^{+\infty}(A)\subset \mathcal{L}_{y}^{+\infty}$. \end{definition}

\begin{remark}\label{R4.5} Let $A\subset X$, $B\subset Y$ and $h:A\to B$ be a weak topological conjugation. Assume that for some $y\in B$ there is $x\in A\setminus M_{X}$ such that $h(x)=y$. Thus, we have $\mathcal{L}_{x}^{+\infty}(A)\subset \mathcal{L}_{y}^{+\infty}$. Therefore, $y$ is  comparable with $x$ with respect to the set $A$.
\end{remark}

It is reasonable to ask whether the converse of Remark \ref{R4.5} holds, that is, if the concept of comparability
implies in the construction of a weak topological conjugation. We shall establish sufficient conditions to show this result.
In order to do that, we first exhibit some auxiliary lemmas.

\begin{lemma}\label{L1.1}
Let $x,q\in X$ and $y\in Y$ be such that $\mathcal{L}_{x,q}^{+\infty}\subset \mathcal{L}_{y}^{+\infty}$. If $\mathcal{L}_{x,q}^{+\infty}\neq \emptyset$,
then there exists a unique point $p\in \tilde{L}_{Y}^{+}(y)$ such that $\mathcal{L}_{x,q}^{+\infty}\subset \mathcal{L}_{y,p}^{+\infty}$.
\end{lemma}
\begin{proof}
Let $\{t_{n}\}_{n\in\N}\in \mathcal{L}_{x,q}^{+\infty}$. By hypothesis we have  $\{t_{n}\}_{n\in\N}\in  \mathcal{L}_{y}^{+\infty}$, that is,
there is a unique point $p\in Y$ such that $$p=\lim_{n\to+\infty}\tilde{\sigma}(y,t_{n}).$$
Consequently, $p\in \tilde{L}_{Y}^{+}(y)$ as $t_{n} \stackrel{n\to+\infty}{\longrightarrow} +\infty$.

Now, suppose to the contrary that there is a sequence $\{s_{n}\}_{n\in \N}\in \mathcal{L}_{x,q}^{+\infty}\setminus \mathcal{L}_{y,p}^{+\infty}$.  
Using the hypothesis, we get $\{s_{n}\}_{n\in\N}\in \mathcal{L}_{y}^{+\infty}$. Thus, there is $\overline{p}\in Y$ $(\overline{p}\neq p)$ such that
 $\overline{p}=\displaystyle\lim_{n\to+\infty}\tilde{\sigma}(y,s_{n})$. 
Define the sequence $\{\overline{t}_n\}_{n\in\mathbb{N}}$ by
\[
\overline{t}_{n}=\left \{\begin{array}{lcl} t_{n},& & n=2k, \quad k=1,2,\ldots,\vspace{0.5mm}\\
s_{n},& & n=2k+1, \quad k = 0,1,2,\ldots\vspace{0.5mm}.\\
\end{array}\right.
\]
By the construction, we obtain $\{\overline{t}_{n}\}_{n\in\N}\in \mathcal{L}_{x,q}^{+\infty}$ and $\{\overline{t}_{n}\}_{n\in\N}\not\in \mathcal{L}_{y}^{+\infty}$ which is a contradiction. Therefore, $\mathcal{L}_{x,q}^{+\infty}\subset \mathcal{L}_{y,p}^{+\infty}$.
\end{proof}

\begin{remark}\label{R0.1} Let $x\in X$ and $A\subset X$. The condition $\tilde{L}_{X}^{+}(x)\cap A\neq \emptyset$ is equivalent to $\mathcal{L}_{x}^{+\infty}(\tilde{L}^{+}_{X}(x)\cap A)\neq \emptyset$.
\end{remark}

\begin{lemma}\label{R1.1} Assume that $y\in Y$ is comparable with $x\in X$ with respect to the nonempty set $\tilde{L}^{+}_{X}(x)\cap A$. Then there is a continuous
 mapping $h:\tilde{L}^{+}_{X}(x)\cap A\to  \tilde{L}_{Y}^{+}(y)$ such that $\mathcal{L}_{x,q}^{+\infty}\subset \mathcal{L}_{y, h(q)}^{+\infty}$ for all $q \in \tilde{L}^{+}_{X}(x)\cap A$. Moreover, if $\{s_{n}\}_{n\in\N}\in \mathcal{L}_{x,q}^{+\infty}$ for $q \in \tilde{L}^{+}_{X}(x)\cap A$ then 
\begin{equation}\label{hfunction}
h(q)=\lim_{n\to+\infty}\tilde{\sigma}(y,s_{n}).
\end{equation}
\end{lemma}
\begin{proof} By hypothesis and Remark \ref{R0.1}, we have $\emptyset \neq \mathcal{L}_{x}^{+\infty}(\tilde{L}^{+}_{X}(x)\cap A)\subset \mathcal{L}_{y}^{+\infty}$. For each point $q\in \tilde{L}^{+}_{X}(x)\cap A$ there is a unique point $p\in \tilde{L}_{Y}^{+}(y)$ such that $\mathcal{L}_{x,q}^{+\infty}\subset \mathcal{L}_{y,p}^{+\infty}$, see  \lemaref{L1.1}. Hence, we may define a mapping $h:\tilde{L}^{+}_{X}(x)\cap A\to  \tilde{L}_{Y}^{+}(y)$ by $h(q)=p$. Consequently, $\mathcal{L}_{x,q}^{+\infty}\subset \mathcal{L}_{y, h(q)}^{+\infty}$ for all $q \in \tilde{L}^{+}_{X}(x)\cap A$. In addition, if $\{s_{n}\}_{n\in\N}\in \mathcal{L}_{x,q}^{+\infty}$ then $\{s_{n}\}_{n\in\N}\in  \mathcal{L}_{y, h(q)}^{+\infty}$, that is, the condition \eqref{hfunction} holds.

In order to show the continuity of $h$, let $q\in \tilde{L}^{+}_{X}(x)\cap A$ and $\{q_{k}\}_{k\in\N}\subset\tilde{L}^{+}_{X}(x)\cap A$ be a sequence such that  $q_{k}\stackrel{k\to+\infty}{\longrightarrow} q$. For each $k\in\mathbb{N}$, there is a sequence $\{s_{n}^{k}\}_{n\in \N}\in \mathcal{L}_{x,q_{k}}^{+\infty}$ since $\{q_{k}\}_{k\in\N}\subset\tilde{L}^{+}_{X}(x)$. 
Now, since $\mathcal{L}_{x,q_k}^{+\infty}\subset \mathcal{L}_{y, h(q_k)}^{+\infty}$ and we have \eqref{hfunction}, we obtain 
\[
h(q_{k})=\displaystyle\lim_{n\to+\infty}\tilde{\sigma}(y,s_{n}^{k}).
\]
Take $p \in  \tilde{L}_{Y}^{+}(y)$ such that $h(q) = p$ and set $p_{k}=h(q_{k})$ for all $k\in\mathbb{N}$.  
Let $\{\alpha_{k}\}_{k\in\N}\subset\R_{+}$ be a sequence such that $\alpha_{k} \stackrel{k\to+\infty}{\longrightarrow} 0$. Then there is $n_{k}\in\mathbb{N}$, $n_{k}> k$, such that 
   $$\rho_{Y}(\tilde{\sigma}(y,s_{n_{k}}^{k}),p_{k})<\alpha_{k} \quad \mbox{and} \quad \rho_{X}(\tilde{\pi}(x,s_{n_{k}}^{k}),q_{k})<\alpha_{k},$$ for all $k\in\mathbb{N}$.
Denote $s_{k}'=s_{n_{k}}^{k}$, $k\in\mathbb{N}$. Thus $\{s_{k}'\}_{k\in\N}\in \mathcal{L}_{x,q}^{+\infty}$ as
$$\rho_{X}(\tilde{\pi}(x,s_{k}'),q)\leq \rho_{X}(\tilde{\pi}(x,s_{k}'),q_{k})+\rho_{X}(q_{k},q)\stackrel{k \rightarrow +\infty}{\longrightarrow} 0.$$
According to the inclusion $\mathcal{L}_{x,q}^{+\infty}\subset \mathcal{L}_{y,p}^{+\infty}$ and by the fact that $\{s_{k}'\}_{k\in\N}\in \mathcal{L}_{x,q}^{+\infty}$, we conclude that
  $$\rho_{Y}(p_{k},p)\leq \rho_{Y}(p_{k},\tilde{\sigma}(y,s_{k}'))+ \rho_{Y}(\tilde{\sigma}(y,s_{k}'),p)<\alpha_{k}+\rho_{Y}(\tilde{\sigma}(y,s_{k}'),p)\stackrel{k \rightarrow +\infty}{\longrightarrow} 0.$$
Hence, $h$ is continuous in $\tilde{L}^{+}_{X}(x)\cap A$.
\end{proof}

\begin{theorem}\label{T2.1} 
Let $x\in X$, $y\in Y$, $A\subset X$ be such that $A\setminus M_{X}$ is positively $\tilde{\pi}$-invariant and assume that $\tilde{L}^{+}_{X}(x)\cap A\neq \emptyset$.
If $y$ is comparable with $x$ with respect to the set $\tilde{L}^{+}_{X}(x)\cap A$, then there is a continuous mapping $h:\tilde{L}^{+}_{X}(x)\cap A\to \tilde{L}^{+}_{Y}(y)$ such that $\mathcal{L}_{x,q}^{+\infty}\subset \mathcal{L}_{y, h(q)}^{+\infty}$ for all $q \in \tilde{L}^{+}_{X}(x)\cap A$ and
\begin{equation}\label{eqt1}
h(\tilde{\pi}(q,t))=\tilde{\sigma} (h(q),t)
\end{equation}
for all $q\in (\tilde{L}^{+}_{X}(x)\cap A)\setminus M_{X}$ and $t\in \mathbb{R}_{+}$. Moreover, if $\tilde{L}^{+}_{X}(x)\cap A$ is $I_{X}$-invariant then
\begin{equation}\label{eqt2}
h(\tilde{\pi}(I_{X}(q),t))=\tilde{\sigma} (h(I_{X}(q)),t)
\end{equation}
for all $q \in \tilde{L}^{+}_{X}(x)\cap A \cap M_X$ and $t\in \mathbb{R}_{+}$.
\end{theorem}
\begin{proof}  By Lemma \ref{R1.1}, there is a continuous mapping $h:\tilde{L}^{+}_{X}(x)\cap A\to  \tilde{L}_{Y}^{+}(y)$ satisfying the condition $\mathcal{L}_{x,q}^{+\infty}\subset \mathcal{L}_{y, h(q)}^{+\infty}$ for all $q \in \tilde{L}^{+}_{X}(x)\cap A$. Further, if $\{s_{n}\}_{n\in\N}\in \mathcal{L}_{x,q}^{+\infty}$ for $q \in \tilde{L}^{+}_{X}(x)\cap A$ then 
\[
h(q)=\lim_{n\to+\infty}\tilde{\sigma}(y,s_{n}).
\]

Let us prove that equality \eqref{eqt1} holds. Let $q\in (\tilde{L}^{+}_{X}(x)\cap A)\setminus M_{X}.$ Then 
\[
\tilde{\pi}(q,t)\in (\tilde{L}^{+}_{X}(x)\cap A)\setminus M_{X}
\]
for all $t\geq 0$ as $(\tilde{L}^{+}_{X}(x)\cap A)\setminus M_{X}$ is positively $\tilde{\pi}$-invariant. Take $\{s_{n}\}_{n\in \N}\in \mathcal{L}_{x,q}^{+\infty}$ and recall that $\mathcal{L}_{x,q}^{+\infty}\subset \mathcal{L}_{y, h(q)}^{+\infty}$. 

\vspace{.2cm}
\textbf{Case 1:} $h(q)\not\in M_{Y}$.

Since $q\not\in M_{X}$ and $h(q)\not\in M_{Y}$, it follows by \lemaref{lema1} and \lemaref{lema2} that there is a sequence
$\{\alpha_{n}\}_{n\in\N}\subset \R_{+}$ such that $\alpha_{n} \stackrel{n\to+\infty}{\longrightarrow} 0$,  
\begin{equation}\label{eq3.3T3.5}
\tilde{\pi}(x,s_{n}+t+\alpha_{n})  \stackrel{n\to+\infty}{\longrightarrow} \tilde{\pi}(q,t)
\end{equation}
and
\begin{equation}\label{eq3.4T3.5}
\tilde{\sigma}(y,s_{n}+t+\alpha_{n})  \stackrel{n\to+\infty}{\longrightarrow} \tilde{\sigma}(h(q),t).
\end{equation}

On the other hand, using \eqref{eq3.3T3.5} and the definition of $h$, we have 
\begin{equation}\label{eq4.10a}
\{s_{n}+t+\alpha_{n}\}_{n\in\N}\in \mathcal{L}_{x,\tilde{\pi}(q,t)}^{+\infty}\subset \mathcal{L}_{y,h(\tilde{\pi}(q,t))}^{+\infty}.
\end{equation}
Therefore, by \eqref{eq3.4T3.5} and \eqref{eq4.10a} we may conclude that
\[ 
h( \tilde{\pi}(q,t))=\lim_{n\to+\infty}\tilde{\sigma}(y,s_{n}+t+\alpha_{n})=\tilde{\sigma}(h(q),t).
\]

\vspace{.2cm}
\textbf{Case 2:}  $h(q)\in M_{Y}$.

Since $M_{Y}$ satisfies the condition STC, there is a STC-tube $F(L_{h(q)},[0,2\lambda])$ through
$h(q)$ given by a section $S_{h(q)}$. Moreover, since the tube is a
neighborhood of $h(q)$, there is $\eta>0$ such that
$$B_{Y}(h(q),\eta)\subset F(L_{h(q)},[0,2\lambda]).$$
Denote $H_1=F(L_{h(q)},(\lambda,2\lambda])\cap B_{Y}(h(q),\eta)$ and $H_2=F(L_{h(q)},[0,\lambda])\cap B_{Y}(h(q),\eta)$, where $B_{Y}(h(q),\eta)=\{y\in Y: \rho_{Y}(h(q),y)<\eta\}.$

We claim that $\{\tilde{\sigma}(y,s_{n})\}_{n\in\N}$ does not admit any subsequence in $H_{1}$.  In fact, let $\{s_{n_k}\}_{k\in\mathbb{N}}$ be a subsequence of $\{s_n\}_{n\in\mathbb{N}}$ and suppose that $\{\tilde{\sigma}(y,s_{n_k})\}_{k\in\N} \subset H_1$. Since $\{s_{n_{k}}\}_{k\in \N}\in \mathcal{L}_{x,q}^{+\infty}\subset\mathcal{L}_{y,h(q)}^{+\infty}$ we get
$\phi_{Y}(\tilde{\sigma}(y,s_{n_{k}}))  \stackrel{k\to+\infty}{\longrightarrow} 0$. Now, using the continuity of $I_{Y}$ and $\sigma$, we obtain

$$\tilde{\sigma}(y,s_{n_{k}}+\phi
_{Y}(\tilde{\sigma}(y,s_{n_{k}})))  \stackrel{k\to+\infty}{\longrightarrow} I_{Y}(h(q)).$$

Also, since $\{s_{n_{k}}\}_{k\in \N}\in \mathcal{L}_{x,q}^{+\infty}$, $q\not\in M_{X}$ and we have \lemaref{lema1}, 
\[
\tilde{\pi}(x,s_{n_{k}}+\phi
_{Y}(\tilde{\sigma}(y,s_{n_{k}})))  \stackrel{k\to+\infty}{\longrightarrow} q.
\]
Consequently, $\{s_{n_{k}}+\phi
_{Y}(\tilde{\sigma}(y,s_{n_{k}}))\}_{k\in\N}\in \mathcal{L}_{x,q}^{+\infty}\subset \mathcal{L}_{y,h(q)}^{+\infty}$. Therefore,
\[ h( q)=\lim_{k\to+\infty}\tilde{\sigma}(y,s_{n_{k}}+\phi
_{Y}(\tilde{\sigma}(y,s_{n_{k}})))=I_{Y}(h(q))
\]
which contradicts the condition (H2).

Hence, we may assume that $\{\tilde{\sigma}(y,s_{n})\}_{n\in\N}\subset  H_{2}$ and we obtain the equality \eqref{eqt1} using a similar proof of Case 1.

In order to prove that condition \eqref{eqt2} holds, it is enough to note that $I_{X}(q) \in (\tilde{L}^{+}_{X}(x) \cap A)\setminus M_X$ for all $q \in \tilde{L}^{+}_{X}(x)\cap A \cap M_X$ and apply condition \eqref{eqt1}.
\end{proof}

\begin{corollary}\label{CT2.1} 
Let $x\in X$, $y\in Y$ and assume that $\tilde{L}^{+}_{X}(x)\neq \emptyset$.
If $y$ is comparable with $x$ with respect to the set $\tilde{L}^{+}_{X}(x)$, then there is a continuous 
mapping $h:\tilde{L}^{+}_{X}(x)\to \tilde{L}^{+}_{Y}(y)$ such that $\mathcal{L}_{x,q}^{+\infty}\subset \mathcal{L}_{y, h(q)}^{+\infty}$ for all $q \in \tilde{L}^{+}_{X}(x)$  and
\[
h(\tilde{\pi}(q,t))=\tilde{\sigma} (h(q),t)
\]
for all $q\in \tilde{L}^{+}_{X}(x)\setminus M_{X}$ and $t\in \mathbb{R}_{+}$. 
\end{corollary}

\begin{corollary}\label{CT2.2} 
Let $x\in X$, $y\in Y$ and assume that $\tilde{L}^{+}_{X}(x)$ is a nonempty $I_{X}$-invariant set.
If $\mathcal{L}_{x}^{+\infty}=\mathcal{L}_{y}^{+\infty}$, then there is a weak topological conjugation $h:\tilde{L}^{+}_{X}(x)\to \tilde{L}^{+}_{Y}(y)$. 
\end{corollary}

\begin{proof}
The condition $\mathcal{L}_{x}^{+\infty}=\mathcal{L}_{y}^{+\infty}$ implies that  $y$ is comparable with $x$ with respect to the set $X$ and $x$ is comparable with $y$ with respect to the set Y. 
By \teoref{T2.1} there are continuous maps $h:\tilde{L}^{+}_{X}(x)\to \tilde{L}^{+}_{Y}(y)$ and $g:\tilde{L}^{+}_{Y}(y)\to \tilde{L}^{+}_{X}(x)$ such that
\[ 
h(q)=\lim_{n\to+\infty}\tilde{\sigma}(y,s_{n}) \quad \mbox{for } q=\lim_{n\to+\infty}\tilde{\pi}(x,s_{n})
\] 
and
\[ 
g(p)=\lim_{n\to+\infty}\tilde{\pi}(x,t_{n}) \quad \mbox{for } p=\lim_{n\to+\infty}\tilde{\sigma}(y,t_{n}).
\] 
It is not difficult to see that $g=h^{-1}$ and $h$ is a homeomorphism. As $\tilde{L}^{+}_{X}(x)$ is $I_{X}$-invariant, we conclude by \teoref{T2.1} that $h$ satisfies \eqref{eqt1} and  \eqref{eqt2}. Therefore, the mapping $h$ defines a weak topological conjugation.
\end{proof}

\subsection{Asymptotically almost periodic motions and weak topological conjugation}

\hspace{0.5cm} In this section, we consider the class of asymptotically almost periodic motions and we investigate some relations between comparable points and the existence of a weak topological conjugation.

\begin{definition} A point $y\in Y$ is said to be strongly comparable with $x\in X$, if  $\mathcal{L}_{x,x}^{+\infty}\subset \mathcal{L}_{y,y}^{+\infty}$ and $\mathcal{L}_{x}^{+\infty}\subset\mathcal{L}_{y}^{+\infty}$.
\end{definition}

\begin{theorem} \label{T2.2}
Let $x\in X\setminus M_{X}$ be almost $\tilde{\pi}$-periodic
and $y\in Y$. Then the following statements are equivalent:
\begin{enumerate}
\item[$a)$] $y\in Y$ is strongly comparable with $x\in X$;
\item[$b)$] there is a continuous mapping  $h:\overline{\tilde{\pi}^{+}(x)}\to \overline{\tilde{\sigma}^{+}(y)}$ satisfying
\[
h(\tilde{\pi}(q,t))=\tilde{\sigma} (h(q),t)
\]
for all  $q\in \overline{\tilde{\pi}^{+}(x)}\setminus M_{X}$ and $t\in \mathbb{R}_{+}$, with $h(x)=y$;
\item[$c)$] $\mathcal{L}_{x}\subset \mathcal{L}_{y}$.
\end{enumerate}
\end{theorem}
\begin{proof} First, let us prove that $(a)\Rightarrow (b)$. Since $x\in X\setminus M_{X}$ is almost $\tilde{\pi}$-periodic, we have  
   $\tilde{L}_{X}^{+}(x)=\overline{\tilde{\pi}^{+}(x)}$ (see Theorem \ref{EG2}). But $y$ is strongly comparable with $x$,  which means that $y$ is positively Poisson $\tilde{\sigma}$-stable. Note that $\tilde{L}_{Y}^{+}(y)\subset \overline{\tilde{\sigma}^{+}(y)}$.

Moreover, since $\mathcal{L}_{x}^{+\infty}(\tilde{L}_{X}^{+}(x)) \subset \mathcal{L}_{x}^{+\infty}\subset\mathcal{L}_{y}^{+\infty}$, we have $y$ is comparable with $x$ 
with respect to the set $\tilde{L}_{X}^{+}(x)$. Thus, according to Corollary \ref{CT2.1} there is a continuous mapping $h:\tilde{L}^{+}_{X}(x)\to \tilde{L}^{+}_{Y}(y)$, that is, $h:\overline{\tilde{\pi}^{+}(x)}\to \overline{\tilde{\sigma}^{+}(y)}$ such that $\mathcal{L}_{x,q}^{+\infty}\subset \mathcal{L}_{y, h(q)}^{+\infty}$ for all $q \in \tilde{L}^{+}_{X}(x)$  and $h(\tilde{\pi}(q,t))=\tilde{\sigma} (h(q),t)$ for all $t\in \mathbb{R}_{+}$ and $q\in \overline{\tilde{\pi}^{+}(x)}\setminus M_{X}$. Note that $x\in \overline{\tilde{\pi}^{+}(x)}=\tilde{L}^{+}_{X}(x)$, consequently there is a sequence $\{t_{n}\}_{n\in\N}\in \mathcal{L}_{x,x}^{+\infty}\subset \mathcal{L}_{y, y}^{+\infty}$ and we obtain
$\displaystyle\lim_{n\to+\infty}\tilde{\sigma}(y,t_{n})=y$. Since $\mathcal{L}_{x,x}^{+\infty}\subset \mathcal{L}_{y, h(x)}^{+\infty}$, we get
$h(x) = y$.

Now, let us show that $(b)\Rightarrow (c)$. Let $\{t_{n}\}_{n\in\N}\in \mathcal{L}_{x}$. 
Then there is $q\in \overline{\tilde{\pi}^{+}(x)}$ such that 
 $\displaystyle\lim_{n\to+\infty}\tilde{\pi}(x,t_{n})=q$. Since $h$ is a continuous mapping, $h(x)=y$ and we have condition $b)$, we get
 \[
\tilde{\sigma}(y,t_{n}) = \tilde{\sigma}(h(x),t_{n}) = h(\tilde{\pi}(x,t_{n})) \stackrel{n\to+\infty}{\longrightarrow} h(q).
\]
Therefore, $\{t_{n}\}_{n\in\N}\in \mathcal{L}_{y}.$
  
At last, we prove that $(c)\Rightarrow (a)$. It is clear that $\mathcal{L}_{x}^{+\infty}\subset\mathcal{L}_{y}^{+\infty}$ as $\mathcal{L}_{x}\subset \mathcal{L}_{y}$. In this way, we need to show that $\mathcal{L}^{+\infty}_{x,x}\subset \mathcal{L}^{+\infty}_{y,y}$. 
Suppose to the contrary that there is  $\{t_{n}\}_{n\in\N}\in \mathcal{L}_{x,x}^{+\infty}\setminus \mathcal{L}_{y,y}^{+\infty}$. Since $\mathcal{L}_{x}\subset \mathcal{L}_{y}$, there is  $p\in Y$, $p\neq y$, with $\{t_{n}\}_{n\in\N}\in \mathcal{L}_{y,p}^{+\infty}$. Now, we consider the sequence 
\[\overline{t}_{n}=\left \{\begin{array}{lcl} t_{n},& & n=2k, \quad k=1,2,\ldots,\vspace{0.5mm}\\
0,& & n=2k+1, \quad k=0,1,2,\ldots\vspace{0.5mm}.\\
\end{array}\right.
\]
Note that $\{\overline{t}_{n}\}_{n\in\N}\in \mathcal{L}^{+\infty}_{x,x}\subset \mathcal{L}_{x}^{+\infty}$ and $\{\overline{t}_{n}\}_{n\in\N}\not\in \mathcal{L}_{y}^{+\infty}$, which is a contradiction.
\end{proof}

\vspace{.1cm}

Recall that a \textit{time
reparametrization} is a homeomorphism $g:\mathbb{R}_{+}\to \mathbb{R}_{+}$ such that $g(0)=0$. Given $x\in X$, we define the following set
 \begin{eqnarray*}
 \mathcal{P}_{x}&=&\{z\in \tilde{L}_X^{+}(x)\cap \tilde{L}_X^{+}(z):\lim_{t\to+\infty}\rho_X(\tilde{\pi}(x,t),\tilde{\pi}(z,g_{z}(t)))=0, \\ && \text{ for some time reparametrization $g_{z}$}\}.\end{eqnarray*}

The set $ \mathcal{P}_x\setminus M_X$ is positively $\tilde{\pi}$-invariant as shown in the next result.

\begin{lemma}\label{L2.15}  The set $ \mathcal{P}_{x}\setminus M_X$ is positively $\tilde{\pi}$-invariant for each $x\in X$. 
\end{lemma}
\begin{proof} Let $x\in X$ and assume that $ \mathcal{P}_{x}\setminus M_X$ is nonempty. For $z\in  \mathcal{P}_{x}\setminus M_X$ and $s>0$, there are $z\in \tilde{L}_X^{+}(x)\cap \tilde{L}_X^{+}(z)$ and a time reparametization $g_{z}$ such that
	\begin{equation}\label{Eq216.1}\lim_{t\to+\infty}\rho_X(\tilde{\pi}(x,t),\tilde{\pi}(z,g_{z}(t))=0.\end{equation}
	 Since $(\tilde{L}^{+}_X(x)\cap \tilde{L}^{+}_X(z))\setminus M_X$ is positively $\tilde{\pi}$-invariant, we have $\tilde{\pi}(z,s)\in (\tilde{L}_X^{+}(x)\cap \tilde{L}_X^{+}(z))\setminus M_{X}$.  Thus, consider the time reparametrization 
	\[
	G_{s}(t)= \left \{\begin{array}{lcl}
	\frac{t}{g_z^{-1}(s+1)} ,& &  t\in[0,g_z^{-1}(s+1)]
	,\vspace{0.5mm}\\
	g_z(t)-s,& & t>g^{-1}_z(s+1).
	\vspace{0.5mm}
	\end{array}\right.
	\]
Consequently,
\[ 
\rho_X(\tilde{\pi}(x,t),\tilde{\pi}(\tilde{\pi}(z, s),G_{s}(t)))= \rho_X(\tilde{\pi}(x,t),\tilde{\pi}(z,g_{z}(t)) \quad \text{for all} \quad t> g^{-1}_{z}(s+1),
\]	
and by \eqref{Eq216.1}, we conclude that $\tilde{\pi}(z,s)\in  \mathcal{P}_{x}\setminus M_X.$
Therefore, $ \mathcal{P}_{x}\setminus M_X$ is positively  $\tilde{\pi}$-invariant.  
\end{proof}

Next, we present sufficient conditions for the set $ \mathcal{P}_{x}\setminus M_X$ to be nonempty. For that, we recall the concept of asymptotic motions as introduced in \cite{artigo2}.

\begin{definition}
A point $x\in X$ is called asymptotically
$\tilde{\pi}$-stationary (resp., asymptotically
$\tilde{\pi}$-periodic, asymptotically almost $\tilde{\pi}$-periodic, asymptotically Poisson $\tilde{\pi}$-stable) if there exist a stationary (resp., $\tilde{\pi}$-periodic, almost $\tilde{\pi}$-periodic, positively Poisson $\tilde{\pi}$-stable) point $p\in X$ and a reparametrization $g_p$ such that
\[
\lim_{t\to+\infty}\rho_X( \tilde{\pi}(x,t), \tilde{\pi}(p,g_p(t)))=0.
\]
\end{definition}

Note that if $x\in X$ is asymptotically almost $\tilde{\pi}$-periodic then $ \mathcal{P}_{x}\neq \emptyset$.

	\begin{lemma}\label{LemaPx111} Let $x\in X$ be an asymptotically almost $\tilde{\pi}$-periodic point. Then $ \mathcal{P}_{x}\setminus M_X$ is a nonempty set.  
		
	\end{lemma}
	\begin{proof}
		Since $x\in X$ is asymptotically almost $\tilde{\pi}$-periodic, 
		there exist an almost $\tilde{\pi}$-periodic point $p\in X  $ and a reparametrization $g_p$ such that
		\begin{equation}
		\label{EQ1111}
		\lim_{t\to+\infty}\rho_X( \tilde{\pi}(x,t), \tilde{\pi}(p,g_p(t)))=0.
		\end{equation}
		This means that $p\in  \mathcal{P}_{x}$.  If $p\in   \mathcal{P}_{x}\setminus M_{X}$ then the proof is complete.
		But if $p\in M_{X}\cap   \mathcal{P}_{x}$, it follows by Lemma \ref{EM1} that $\tilde{\pi}(p,s)\in X\setminus M_{X}$ is an almost periodic point for all
		$0<s<\phi(p)$. According to Theorem \ref{EG2}, $
		\tilde{\pi}(p,s)\in \tilde{L}_{X}^{+}(\tilde{\pi}(p,s))$ and by \eqref{EQ1111} we have $\tilde{\pi}(p,s)\in  \tilde{L}_{X}^{+}(x)$.
		Using the time reparametrization 
			\[
			G_{s}(t)= \left \{\begin{array}{lcl}
			\frac{t}{g_p^{-1}(s+1)} ,& &  t\in[0,g_p^{-1}(s+1)]
			,\vspace{0.5mm}\\
			g_p(t)-s,& & t>g^{-1}_p(s+1),
			\vspace{0.5mm}
			\end{array}\right.
			\]
		we obtain
\[
\rho_X( \tilde{\pi}(x,t), \tilde{\pi}(\tilde{\pi}(p,s),G_{s}(t)))=
		\rho_X( \tilde{\pi}(x,t),\tilde{\pi}(p,g_{p}(t)) \quad \text{for all} \quad t>g^{-1}_p(s+1).
\]
Consequently, by \eqref{EQ1111}, we have $\tilde{\pi}(p,s)\in  \mathcal{P}_{x}\setminus M_{X}$.
\end{proof}

In the next definition, we define the concept of comparability in limit.

\begin{definition}
A point $y\in Y$ is said to be comparable in limit with $x\in X$, if  $\mathcal{L}_{x}^{+\infty}\subset \mathcal{L}_{y}^{+\infty}$.
\end{definition}

Let $z \in  \mathcal{P}_x$. By definition of $ \mathcal{P}_x$ there is a time reparametrization $g_z$ such that  $$\displaystyle\lim_{n\to+\infty}\rho_{X}(\tilde{\pi}(x,t),\tilde{\pi}(z, g_z(t)))=0.$$
We shall denote this reparametrization by $g_z^x$.

\begin{theorem}\label{T2.3}
 Let $x\in X$ be asymptotically almost  $\tilde{\pi}$-periodic and $y\in Y$. The point $y$ 
 is comparable in limit with $x$ if and only if there is a continuous mapping $h:\tilde{L}_{X}^{+}(x)\to\tilde{L}_{Y}^{+}(y)$ 
satisfying the conditions:
\begin{enumerate}
\item[$a)$] $\mathcal{L}_{x,q}^{+\infty}\subset \mathcal{L}_{y, h(q)}^{+\infty}$ for all $q \in \tilde{L}^{+}_{X}(x)$;

\item[$b)$] $h(\tilde{\pi}(q,t))=\tilde{\sigma} (h(q),t)$ for all $q\in \tilde{L}^{+}_{X}(x)\setminus M_{X}$ and $t\in \mathbb{R}_{+}$;

\item[$c)$] For $\tilde{q}\in  \mathcal{P}_{x}\setminus M_{X}$ and $\{t_{n}\}_{n\in\N}\in \mathcal{L}_{x}^{+ \infty}$ we have
\[
\lim_{n\to+\infty}\rho_{Y}(\tilde{\sigma}(y,t_{n}),\tilde{\sigma}(h(\tilde{q}),g_{\tilde{q}}^{x}(t_{n})))=0.
\]
\end{enumerate}
\end{theorem} 
\begin{proof}
Since $x\in X$ is asymptotically almost  $\tilde{\pi}$-periodic, we have  $ \mathcal{P}_{x}\setminus M_{X} \neq \emptyset$, see Lemma \ref{LemaPx111}.

 First, let us prove the necessary condition. Note that $\tilde{L}^{+}_{X}(x) \neq \emptyset$ as $x\in X$ is asymptotically almost $\tilde{\pi}$-periodic. Besides, since $y$ is comparable in limit with $x$ we have
 \[
 \mathcal{L}_{x}^{+\infty}(\tilde{L}^{+}_{X}(x)) \subset  \mathcal{L}_{x}^{+\infty}\subset \mathcal{L}_{y}^{+\infty}.
 \]
 By Corollary \ref{CT2.1}, there is a continuous mapping $h:\tilde{L}^{+}_{X}(x)\to \tilde{L}^{+}_{Y}(y)$ such that $\mathcal{L}_{x,q}^{+\infty}\subset \mathcal{L}_{y, h(q)}^{+\infty}$ for all $q \in \tilde{L}^{+}_{X}(x)$  and $h(\tilde{\pi}(q,t))=\tilde{\sigma} (h(q),t)$ for all $q\in \tilde{L}^{+}_{X}(x)\setminus M_{X}$ and $t\in \mathbb{R}_+$. 

 Now, let $\{t_{n}\}_{n\in\N}\in \mathcal{L}_{x}^{+\infty}$ and $\tilde{q}\in  \mathcal{P}_{x}\setminus M_{X}$. By definition of $ \mathcal{P}_x$ there is a time reparametrization $g_{\tilde{q}}^{x}$ satisfying
 \begin{equation}\label{eq3.10}
 \lim_{t\to+\infty}\rho_{X}( \tilde{\pi}(x,t), \tilde{\pi}(\tilde{q},g_{\tilde{q}}^{x}(t)))=0.
 \end{equation}
 Since $\{t_{n}\}_{n\in\N}\in \mathcal{L}_{x}^{+\infty}$, there exists $\overline{q}\in \tilde{L}_{X}^{+}(x)$ such that
\begin{equation}\label{eq3.9}\overline{q}=\lim_{n\to+\infty}\tilde{\pi}(x,t_{n})\quad \mbox{and}\quad h(\overline{q})=\lim_{n\to+\infty}\tilde{\sigma}(y,t_{n}),
\end{equation}
where the second limit presented above is obtained by the condition $\mathcal{L}_{x,q}^{+\infty}\subset \mathcal{L}_{y, h(q)}^{+\infty}$ for all $q \in \tilde{L}^{+}_{X}(x)$. By \eqref{eq3.10} and \eqref{eq3.9}, we have $\overline{q}=\displaystyle\lim_{n\to+\infty}\tilde{\pi}(\tilde{q},g_{\tilde{q}}^{x}(t_{n}))$.
 By continuity of $h$, we get
  \begin{equation}\label{eq3.10.1}
 h(\overline{q})=\displaystyle\lim_{n\to+\infty}h(\tilde{\pi}(\tilde{q},g_{\tilde{q}}^{x}(t_{n}))).
 \end{equation}

Using condition $b)$ and \eqref{eq3.10.1}, we obtain
\[
h(\overline{q}) =\displaystyle\lim_{n\to+\infty}\tilde{\sigma}(h(\tilde{q}),g_{\tilde{q}}^{x}(t_{n})).
\]
Hence,
\[
\rho_{Y}(\tilde{\sigma}(y,t_{n}),\tilde{\sigma}(h(\tilde{q}),g_{\tilde{q}}^{x}(t_{n}))) \stackrel{n\to+\infty}{\longrightarrow} 0
\]
and the necessary condition is proved.

Conversely, let $\{t_{n}\}_{n\in\N}\in\mathcal{L}_{x}^{+\infty}$. Then there is $q\in\tilde{L}_{X}^{+}(x)$ such that $\tilde{\pi}(x,t_{n})\stackrel{n\to+\infty}{\longrightarrow} q$. By \lemaref{LemaPx111}, we may take
$\tilde{q}\in  \mathcal{P}_{x}\setminus M_{X}$. Let $g_{\tilde{q}}^{x}$ be the time reparametrization such that 
\[
 \lim_{t\to+\infty}\rho_{X}( \tilde{\pi}(x,t), \tilde{\pi}(\tilde{q},g_{\tilde{q}}^{x}(t)))=0.
 \]
Hence, $\tilde{\pi}(\tilde{q},g_{\tilde{q}}^{x}(t_{n})) \stackrel{n\to+\infty}{\longrightarrow} q$. Using condition $b)$, we have 
\[h(q)=h(\lim_{n\to+\infty}\tilde{\pi}(\tilde{q},g_{\tilde{q}}^{x}(t_{n})))=\lim_{n\to+\infty}h(\tilde{\pi}(\tilde{q},g_{\tilde{q}}^{x}(t_{n})))=
\lim_{n\to+\infty}\tilde{\sigma}(h(\tilde{q}),g_{\tilde{q}}^{x}(t_{n})).\]
Since
 $$\rho_{Y}(\tilde{\sigma}(y,t_{n}),h(q))\leq \rho_{Y}(\tilde{\sigma}(y,t_{n}),\tilde{\sigma}(h(\tilde{q}),g_{\tilde{q}}^{x}(t_{n}))+\rho_{Y}(\tilde{\sigma}(h(\tilde{q}),g_{\tilde{q}}^{x}(t_{n})),h(q)),$$
and we have condition $c)$, we conclude that $\{t_{n}\}_{n\in\N}\in \mathcal{L}_{y}^{+\infty}$, i.e., $\mathcal{L}_{x}^{+\infty}\subset \mathcal{L}_{y}^{+\infty}$ and the result is proved.
\end{proof}

\begin{corollary}\label{C3.12} Let $x\in X$ be asymptotically almost  $\tilde{\pi}$-periodic such that $\overline{\tilde{\pi}^{+}(x)}$ is compact.
The point $y\in Y$ 
is comparable in limit with $x$ if and only if there is a continuous mapping $h:\tilde{L}_{X}^{+}(x)\to\tilde{L}_{Y}^{+}(y)$ satisfying the conditions:
\begin{enumerate}
\item[$a)$] $\mathcal{L}_{x,q}^{+\infty}\subset \mathcal{L}_{y, h(q)}^{+\infty}$ for all $q \in \tilde{L}^{+}_{X}(x)$;

\item[$b)$] $h(\tilde{\pi}(q,t))=\tilde{\sigma} (h(q),t)$ for all $q\in \tilde{L}^{+}_{X}(x)\setminus M_{X}$ and $t\in \mathbb{R}_{+}$;

\item[$c)$] For $\tilde{q}\in  \mathcal{P}_{x}\setminus M_{X}$ we have
\[
\lim_{t\to+\infty}\rho_{Y}(\tilde{\sigma}(y,t),\tilde{\sigma}(h(\tilde{q}),g_{\tilde{q}}^{x}(t)))=0.
\]
\end{enumerate}
\end{corollary}
\begin{proof} It is enough to prove condition $c)$ in the necessary condition. As showed in the proof of Theorem \ref{T2.3}, let $h:\tilde{L}_{X}^{+}(x)\to\tilde{L}_{Y}^{+}(y)$ be a continuous mapping satisfying conditions $a)$ and $b)$. Suppose to the contrary that
there are $\epsilon_{0}>0$, $q\in  \mathcal{P}_{x}
\setminus M_{X}$, $g_{q}^{x}$ a time reparametrization 
 and $\{t_{n}\}_{n\in\N}\subset \mathbb{R}_{+}$ a sequence with $t_{n}\stackrel{n\to+\infty}{\longrightarrow} +\infty$ such that
  \[ 
  \rho_{Y}(\tilde{\sigma}(y,t_{n}),\tilde{\sigma}(h(q),g_{q}^{x}(t_{n})))\geq \epsilon_0,
  \]
for all $n\in\mathbb{N}$. By compactness of $\overline{\tilde{\pi}^{+}(x)}$, there is a  subsequence $\{t_{n_{k}}\}_{k\in\N}$ of $\{t_{n}\}_{n\in\N}$ such that $\{t_{n_{k}}\}_{k\in\N}\in\mathcal{L}_{x}^{+\infty}\subset
\mathcal{L}_{y}^{+\infty}$ and according to the proof of \teoref{T2.3}, we conclude that
 \[\lim_{k\to+\infty}\rho_{Y}(\tilde{\sigma}(y,t_{n_{k}}),\tilde{\sigma}(h(q),g_{q}^{x}(t_{n_{k}})))=0\]
which is a contradiction.
\end{proof}

  \begin{theorem}\label{T2.14}\rm\cite[Theorem 3.14]{artigo2}\it  \, Let $(X,\pi;M,I)$ be an impulsive system and
$X$ be a complete metric space. If $x\in X$ is asymptotically
almost $\tilde{\pi}$-periodic,
 then:
\begin{enumerate}
\item[$a)$] $\overline{\tilde{\pi}^{+}(x)}$ is compact;
\item[$b)$] $\tilde{L}_X^{+}(x)$ coincides with the closure of an
almost $\tilde{\pi}$-periodic orbit.
\end{enumerate}
\end{theorem}
  
  If $x \in X$ satisfies some asymptotic property with $X$ complete and $y\in Y$ is comparable in limit with $x$ then $y$ also satisfies this asymptotic property. See the next result.
  
  \begin{theorem}\label{T2.4}
  Let $X$ be a complete metric space and $x\in X$ be a asymptotically almost $\tilde{\pi}$-periodic
(asymptotically $\tilde{\pi}$-stationary, asymptotically
$\tilde{\pi}$-periodic) point. If $y\in Y$ is comparable in limit with $x$, then $y$ is also  
asymptotically almost $\tilde{\sigma}$-periodic (asymptotically $\tilde{\sigma}$-stationary, asymptotically
$\tilde{\sigma}$-periodic).
 \end{theorem}
 \begin{proof} Suppose that $x\in X$ is 
  asymptotically almots $\tilde{\pi}$-periodic and $y\in Y$ is comparable in limit with $x$. Theorem \ref{T2.14}  states 
 the compactness of $\overline{\tilde{\pi}^{+}(x)}$ and implies that $\tilde{L}_{X}^{+}(x)=\overline{\tilde{\pi}^{+}(p)}$ for some point $p$ almost $\tilde{\pi}$-periodic.
 
By \corref{C3.12}, there is a uniformly continuous mapping $h:\tilde{L}_{X}^{+}(x)\to\tilde{L}_{Y}^{+}(y)$ satisfying the conditions $a)$, $b)$ and $c)$ from Corollary \ref{C3.12}. 
 The uniform continuity of $h$ implies that for $\epsilon>0$ given, there is $\delta>0$ such that  
\begin{equation}\label{eq3.14} \rho_{Y}(h(z_1),h(z_2))<\epsilon\quad \text{whenever} \quad \rho_{X}(z_1,z_2)<\delta.\end{equation}

 By \lemaref{LemaPx111} we have $\mathcal{P}_{x}\setminus M_{X}\neq\emptyset$. Let $q\in  \mathcal{P}_{x}\setminus M_{X}$, then $q\in \tilde{L}_{X}^{+}(x)\setminus M_{X}=\overline{\tilde{\pi}^{+}(p)}\setminus M_{X}$. According  to Theorem \ref{EG1}
 the point $q$ is almost $\tilde{\pi}$-periodic, that is, for a given $\delta>0$, there is $T=T(\delta)>0$ such that for any $\alpha\geq0$, the interval $[\alpha,\alpha+T]$ contains a number
$\tau_{\alpha}>0$ such that
\begin{equation}\label{eq3.15}
\rho_{X}(\tilde{\pi}(q,t+\tau_{\alpha}), \tilde{\pi}(q,t))<\delta \quad \mbox{\mbox{for all} }   t\geq 0.\end{equation}

By \eqref{eq3.14} and \eqref{eq3.15},

\[
\rho_{Y}(\tilde{\sigma}(h(q),t+\tau_{\alpha}),\tilde{\sigma}(h(q),t))=\rho_{Y}(h(\tilde{\pi}(q,t+\tau_{\alpha})),h(\tilde{\pi}(q,t)))<\epsilon
\]
for all $t\geq 0$. It means that $h(q)$ is almost $\tilde{\sigma}$-periodic. Using condition $c)$ from Corollary \ref{C3.12}, we conclude that $y$ is asymptotic almost $\tilde{\sigma}$-periodic.
\end{proof}

\begin{definition} A mapping $h:X\to Y$ is called an \emph{$I$-homomorphism} from the impulsive system $(X,\pi; M_{X}, I_{X})$ taking values in $(Y,\sigma; M_{Y}, I_{Y})$, if the following conditions hold:
\begin{itemize} 
\item[$a)$] $h$ is continuous in X;
\item[$b)$] $h(\tilde{\pi}(x,t))=\tilde{\sigma}(h(x),t)$ for all $x\in  X\setminus M_{X}$ and $t\geq 0$;
\item[$c)$] $h(\tilde{\pi}(I_{X}(x),t))=\tilde{\sigma}(h(I_{X}(x)),t)$ for all $x\in  X\cap M_{X}$ and $t\geq 0$.
\end{itemize}
\end{definition} 

Let $h:X\to Y$ be an $I$-homomorphism and $y\in Y$.
Consider the set
\begin{equation}\label{eq3.17}
X_{y}=\{x\in X: h(x)=y\}.
\end{equation}

In \obsref{R4.5}, we pointed out that $y$ is always comparable in limit with $x\in X_{y}\setminus M_{X}$.  In the next result we concern of with the converse, that is, 
when $x$ is comparable in limit with $y$. 

\begin{theorem}\label{T3.15} Let $h:X\to Y$  be an $I$-homomorphism, $y\in Y$ and $x\in X_{y}\setminus M_{X}$ be a point such that $\overline{\tilde{\pi}^{+}(x)}$ is compact.  Assume that for each point $q\in \tilde{L}_{Y}^{+}(y)$ there is a unique $p\in X$ such that $h(p)=q$.
Then $x$ is comparable in limit with $y$.
\end{theorem}
\begin{proof}
Let  $\{t_{n}\}_{n\in\N}\in \mathcal{L}_{y}^{+\infty}$. Then there is $q\in \tilde{L}_{Y}^{+}(y)$ such that
\begin{equation}\label{318.1}
\lim_{n\to+\infty}\tilde{\sigma}(y,t_{n})= q.
\end{equation}
Let  $\{\tilde{\pi}(x, t_{n_k})\}_{k\in\N}$ be an arbitrary subsequence of $\{\tilde{\pi}(x, t_n)\}_{n\in\N}$. Since $\overline{\tilde{\pi}^{+}(x)}$ is compact,
the sequence $\{\tilde{\pi}(x, t_{n_k})\}_{k\in\N}$ admits a convergent subsequence $\{\tilde{\pi}(x, t_{n_k}')\}_{k\in\N}$, that is,
\begin{equation}\label{318.2}
\lim_{k\to+\infty}\tilde{\pi}(x, t_{n_{k}}')= p \in \tilde{L}_{X}^{+}(x).
\end{equation} 
On the other hand, since $h$ is an \emph{$I$-homomorphism}, we have
\begin{equation}\label{318.3}
h(p) = \lim_{k\to+\infty}h(\tilde{\pi}(x, t_{n_{k}}'))= \lim_{k\to+\infty}\tilde{\sigma}(h(x), t_{n_{k}}') = \lim_{k\to+\infty}\tilde{\sigma}(y, t_{n_{k}}').
\end{equation} 
By \eqref{318.1} and \eqref{318.3}, we conclude that $h(p)=q$. Thus, every subsequence of $\{\tilde{\pi}(x, t_n)\}_{n\in\N}$ admits a further subsequence which converges to $p$ since $h^{-1}(\{q\})$ is a singleton. Hence, $\{\tilde{\pi}(x, t_n)\}_{n\in\N}$ is convergent and $\{t_{n}\}_{n\in\N}\in \mathcal{L}_{y}^{+\infty}$. This shows that $x$ is comparable in limit with $y$.
\end{proof}

\begin{corollary}
Let $Y$ be a complete metric space, $y\in Y$ and $x\in X_{y}\setminus M_{X}$ be a point such that $\overline{\tilde{\pi}^{+}(x)}$ is compact. Assume that $y$ is asymptotically almost $\tilde{\sigma}$-periodic
(asymptotically $\tilde{\sigma}$-stationary, asymptotically
$\tilde{\sigma}$-periodic) and for each point $q\in \tilde{L}_{Y}^{+}(y)$ there is a unique $p\in X$ such that $h(p)=q$.
Then  $x$ is asymptotically almost $\tilde{\pi}$-periodic
(asymptotically $\tilde{\pi}$-stationary, asymptotically
$\tilde{\pi}$-periodic).  

\end{corollary}
\begin{proof}
This result follows by \teoref{T2.4} and  \teoref{T3.15}.
\end{proof}


\begin{thebibliography}{00}


\bibitem{Ambrosino} Ambrosino, R., Calabrese, F., Cosentino, C. and De Tommasi, G.,  Sufficient conditions
for finite-time stability of impulsive dynamical systems, IEEE Trans. Autom.
Control, 54 (2009), 861-865.



\bibitem{BonottoFederson1} Bonotto, E. M. and  Federson, M., Topological conjugation and asymptotic stability
in impulsive semidynamical systems, J. Math. Anal. Appl.,
326 (2007), 869-881.


\bibitem{artigo2} Bonotto, E. M., Gimenes, L. P. and Souto, G. M., Asymptotically almost periodic
motions in impulsive semidynamical systems, Topological Methods in
Nonlinear Analysis, 49 (2017), 133-163.

\bibitem{BonottoMatheus}  Bonotto, E. M.,  Bortolan, M. C., Carvalho, A. N. and Czaja, R., Global attractors for impulsive dynamical systems - a precompact approach, J. Diff. Equations, (2015), 2602-2625.

\bibitem{Manuel}  Bonotto, E. M. and Jimenez, M. J., On impulsive semidynamical systems: minimal, recurrent and almost periodic motions. Topological Methods in Nonlinear Analysis, 44 (2014), 121-141.

\bibitem{jaqueline}  Bonotto, E. M. and Ferreira, J. C., Dissipativity in impulsive systems via Lyapunov functions, Mathematische Machrichten, 2-3 (2016), 213-231.

\bibitem{Cheban}  Cheban, D. N., Asymptotically Almost Periodic Solutions of Differential
Equations, Hindawi,  Publishing Corporation, 2009.


\bibitem{Ciesielski2} Ciesielski, K., On semicontinuity in impulsive dynamical systems,
Bull. Polish Acad. Sci. Math., 52 (2004), 71-80.

\bibitem{Ciesielski3} Ciesielski, K., On stability in impulsive dynamical systems, Bull. Polish Acad. Sci.
Math., 52 (2004), 81-91.


\bibitem{Cortes} Cort\'{e}s, J., Discontinuous dynamical systems: a tutorial on solutions, nonsmooth
analysis, and stability, IEEE Control Syst. Mag., 28(3) (2008), 36-73.

\bibitem{Dash} Dashkovskiy, S., Feketa, P., Kapustyan, O. and Romaniuk, I., Invariance and stability of global attractors for multi-valued impulsive dynamical systems, J. Math. Anal. Appl., 458 (2018), 193-218.

\bibitem{El} El-Gohary, A. and Al-Ruzaiza, A. S., Chaos and adaptive control in two prey, one
predator system with nonlinear feedback, Chaos Solitons and Fractals, 34 (2007),
443-453.




\bibitem{Irwin} Irwin, M. C., Smooth Dynamical Systems. Academic Press, London, 1980.

\bibitem{Kaul} Kaul, S. K., On impulsive semidynamical systems, J. Math. Anal. Appl., 150 (1),
(1990), 120-128.



\bibitem{Katok} Katok, A. and Hasselblatt, B., Introduction to the modern theory of dynamical systems, Encyclopedia 
of Mathematics and its applications 54, Cambridge University Press, 1995.

\bibitem{Liang} Liang, Z., Zeng, X., Pang, G. and Liang, Y., Periodic solution of a Leslie predator-prey system with ratio-dependent 
and state impulsive feedback control, Nonlinear Dynam., 89 (2017), 2941-2955.

\bibitem{Riva} Rivadeneira, Pablo S. and Moog, Claude H., Observability criteria for impulsive control systems with applications to biomedical 
engineering processes, Automatica J. IFAC, 55 (2015), 125-131.

\bibitem{Smalle} Smale, S., On dynamical systems, Boletin de la Sociedad Matematica, (1960), 195-198.

\bibitem{Yuan} Yuan, Chengzhi and Wu, Fen, Delay scheduled impulsive control for networked control systems, IEEE Trans. Control Netw. Syst., 4 (2017), 587-597.




\end{thebibliography}
 \end{document}